\documentclass[12pt]{article}
\usepackage[utf8]{inputenc}
\usepackage{amsmath, amssymb, amsthm, mathtools,bbm}
\usepackage{bbold,tabularx}
\usepackage{enumerate,float,subcaption,xcolor,multirow}
\usepackage{tikz, tikzpagenodes}
\usetikzlibrary{through,intersections}
\usepackage{tikz-cd}
\mathtoolsset{centercolon}
\usepackage{xurl}

\usepackage{relsize}

\usepackage[T1]{fontenc}
\usepackage{verbatim}
\usepackage[margin=3cm]{geometry}
\usepackage[algo2e,ruled,vlined]{algorithm2e}
\usepackage{algorithm}
\usepackage{here}
\usepackage{setspace}
\usepackage[normalem]{ulem}
\usepackage{algpseudocode}
\usepackage{dsfont}
\usepackage{url}
\usepackage[inline]{enumitem}
\usepackage{microtype}
\usepackage[affil-sl]{authblk}
\usepackage{graphicx}
\usepackage{nicematrix}
\usepackage{amsmath}
\usepackage{amssymb}
\usepackage{amsthm}
\usepackage{mathtools}
\usepackage{nicefrac}
\usepackage{amsfonts,latexsym,graphics}
\usepackage{color}
\usepackage[hidelinks]{hyperref}
\usepackage{adjustbox}
\usepackage{booktabs}
\usepackage{makeidx}
\usepackage{bm}
\usepackage[mathscr]{euscript}

\usepackage{tikz}
\usetikzlibrary{tikzmark}
\usetikzlibrary{positioning}
\usetikzlibrary{graphs}
\usetikzlibrary{graphs.standard}
\usetikzlibrary{decorations.pathmorphing}
\usetikzlibrary{decorations.text}
\pgfdeclarelayer{bg}    
\pgfsetlayers{bg,main}  

\usepackage{xspace}         

\makeatletter
\newcommand{\rref}[2]{\hyperref[#2]{{#1}~\ref*{#2}}}

\usepackage{amsfonts,amsthm,amsmath,mathtools,amssymb,comment,xcolor,enumitem}
\usepackage{mathtools}
\usepackage{soul}

\theoremstyle{plain}
\newtheorem{theorem}{Theorem}[section]
\newtheorem{proposition}[theorem]{Proposition}
\newtheorem{lemma}[theorem]{Lemma}

\theoremstyle{definition}

\newtheorem{conjecture}[theorem]{Conjecture}
\newtheorem{problem}{Problem}

\newcommand\tran{\mkern-2mu\raise1.25ex\hbox{$\scriptscriptstyle\top$}\mkern-3.5mu}

\def\tr{\mathop{\rm tr }\nolimits}

\newcommand{\R}{\mathds{R}}

\newcommand{\F}{\mathds{F}}

\def\matrix0{{\mbox {\boldmath $O$}}}

\def\j{{\mbox{\boldmath $1$}}}
\def\vec0{\mbox{\bf 0}}
\def\vecrho{{\mbox{\boldmath $\rho$}}}
\def\vecnu{{\mbox{\boldmath $\nu$}}}

\title{On the sum of the largest and smallest eigenvalues\\ of odd-cycle free graphs}
\author{Aida Abiad\thanks{\texttt{a.abiad.monge@tue.nl},  Department of Mathematics and Computer Science, Eindhoven University of Technology, The Netherlands}\thanks{Department of Mathematics and Data Science of Vrije Universiteit Brussel, Belgium}\qquad Vladislav Taranchuk\thanks{\texttt{Vlad.Taranchuk@UGent.be}, Department of Mathematics: Analysis, Logic and Discrete Mathematics, Ghent University, Belgium}  \qquad Thijs van Veluw\thanks{\texttt{Thijs.vanVeluw@ugent.be}, Department of Mathematics, Computer Science and Statistics, Ghent University, Belgium}\thanks{Department of Mathematics and Computer Science, Eindhoven University of Technology, The Netherlands}}

\date{}

\begin{document}

\maketitle

\begin{abstract}
    Let $G$ be a graph with adjacency eigenvalues $\lambda_1 \geq \cdots \geq \lambda_n$. Both $\lambda_1 + \lambda_n$ and the odd girth of $G$ can be seen as measures of the bipartiteness of $G$. Csikvári proved in 2022 that for odd girth 5 graphs (triangle-free) it holds that $(\lambda_1+\lambda_n)/n \le (3-2\sqrt 2) < 0.1716$. In this paper we extend Csikvári's result to general odd girth $k$ proving that $(\lambda_1+\lambda_n)/n = O(k^{-1})$. In the case of odd girth 7, we prove a stronger upper bound of $(\lambda_1+\lambda_n)/n < 0.0396$.\\

\end{abstract}

\section{Introduction}

Let $\lambda_1\geq \cdots \geq \lambda_n$ be the adjacency eigenvalues of a simple graph $G$. 
Denote by $Q = A + D$ the \emph{signless Laplacian matrix} of $G$, where $D$ is the
diagonal matrix of the degrees of $G$ and $A$ is the adjacency matrix of $G$. Let $q_n\leq \cdots \leq q_1$ denote the eigenvalues of $Q$. For $d$-regular graphs, $\lambda_1=d$ and $D=d\cdot I$, so $q_n =\lambda_1+\lambda_n$, but for irregular graphs, this is not necessarily true.

Both $q_n$ and $\lambda_1 + \lambda_n$ have been studied as measures for the bipartiteness of a graph, see e.g. \cite{BCLNV2023,B1998,C2022,LNO2016}. Indeed, it is known that a graph $G$ is bipartite if and only if $\lambda_1 + \lambda_n= 0$, and that $G$ has a bipartite connected component if and only if $q_n=0$. Let $D_2(G)$ be the minimum number of edges that have to be removed in order to make $G$ bipartite. Brandt \cite{B1998} and de Lima, Nikiforov and Oliveira \cite{LNO2016} showed that the following inequality holds for any graph
\begin{equation}\label{eq:qn Dn}
    D_2(G)\ge \frac {q_n(G)}{4} n.
\end{equation}
 The interest on the parameter $D_2(G)$ was spurred by an old conjecture of Erd\H{o}s \cite{E1976} which states the following:

\begin{conjecture}[Erd\H{o}s 1976]\label{conj:D2}
    If $G$ is a triangle-free graph on $n$ vertices, then $D_2(G)\le n^2/25$.
\end{conjecture}

The best-known upper bound is currently $D_2(G) \leq n^2/23.5$ which was shown by Balogh, Clemen and
Lidick\'{y} \cite{BCL2021}. Using (\ref{eq:qn Dn}), from Conjecture \ref{conj:D2}, it would follow that $q_n/n\le 4/25 =0.16 $ for triangle-free graphs. This was shown to be true by Balogh, Clemen,
Lidick\'{y}, Norin and Volec \cite{BCLNV2023}:
\begin{theorem}[{\cite[Theorem 1.6]{BCLNV2023}}]
    If $G$ is a triangle-free graph on $n$ vertices, then $q_n(G)/n \le 15/94<0.1596$.
\end{theorem}

However, as is mentioned in \cite{BCLNV2023}, this bound is not sharp, and it can be improved to $0.1547$, and even to $0.1545$ when restricting to regular graphs. The triangle-free graph giving the highest known value of $q_n/n$ is the Higman-Sims graph (the unique strongly regular graph with parameters $(100,22,0,6)$), giving $q_n/n=7/50=0.14$. So in fact, it remains an open question to determine a sharp upper bound for $q_n/n$ over all triangle-free graphs. More generally, there has been recent work which studied the relationship between the largest and smallest eigenvalues of a graph and its clique number \cite{B1998, LNO2016, N2006}. In fact, determining an upper bound for $\lambda_1 + \lambda_n$ in a $K_{r+1}$-free graph has been stated as unsolved problem in \cite{LN2023}.

Recall that for regular graphs, $q_n = \lambda_1 + \lambda_n$. Brandt \cite{B1998} was the first to give an upper bound on $q_n/n$ for triangle-free regular graphs, demonstrating that in this case $q_n/n \leq 3 - 2\sqrt{2} < 0.1716$. Recently, Csikv\'{a}ri \cite{C2022} extended this upper bound of the quantity $(\lambda_1 + \lambda_n)/n$ to all triangle-free graphs. Thus, when $G$ is triangle-free, we have the following inequality.
\begin{equation}\label{eq:3-2sqrt2}
    \frac{\lambda_1+\lambda_n}{n}\le 3-2\sqrt 2.
\end{equation}

While in \cite{BCLNV2023} the stronger bound of 0.1597 has been shown to hold for $q_n/n$, no such improvements have been given for the quantity $(\lambda_1 + \lambda_n)/n$.

It is well-known that a graph $G$ is bipartite if and only if it contains no odd cycles. Hence a natural measure of the bipartiteness of $G$ is given by its \textit{odd girth}, which is the length of the shortest odd cycle in $G$. When $G$ is bipartite, the odd girth is said to be $\infty$. In this paper we investigate precisely the relationship between the following two measures of bipartiteness: $\lambda_1+\lambda_n$ and the odd girth. Our results build on the work of \cite{B1998, C2022} and can be seen as a natural extension of it since the triangle-free case is equivalent to the odd girth 5 case. Define 
$$
\gamma_{k} = \sup\left\{ \frac{\lambda_1 + \lambda_n}{n}: G \text{ has odd girth at least }k \right\}.
$$

The first case of odd girth 3 is attained by any infinite sequence of complete graphs, which implies $\gamma_3 = 1$. The next case of odd girth 5
can equivalently be restated as the triangle-free case. Therefore, using (\ref{eq:3-2sqrt2}) for the upper bound, together with the Higman-Sims graph for the lower bound, it follows that
$$
0.14 \leq \gamma_5 < 0.1716.
$$
In the case that $G$ is regular, $q_n = \lambda_1 + \lambda_n$ and so in fact the stronger bound of $0.1545$ holds \cite{BCLNV2023}. Our first result is a general bound on $\gamma_{k}$.
\begin{theorem}\label{thm:general}
    Let $k$ be an odd positive integer, then there exists constants $c_1$ and $c_2$ such that 
    $$
    c_1k^{-3} \leq \gamma_{k} < c_2k^{-1}
    $$
\end{theorem}

Though Theorem \ref{thm:general} is stated in terms of asymptotics, we do obtain explicit expressions for the upper and lower bounds in Section 3.1 (see Proposition \ref{prop:l1+ln bound odd girth through ratio}). The lower bound is attained by the cycle of odd length $k$. The upper bound is obtained by extending ideas in \cite{C2022}.

Our second result gives much stronger bounds for $\gamma_7$. In fact, our upper bound for $\gamma_7$ also improves the general upper bound for $\gamma_9$, and to do so we used the technique of weight interlacing. Furthermore, we found that the folded 7-cube outperforms the 7-cycle for a lower bound for $\gamma_7$. Our search via the house of graphs database \cite{houseofgraphs} revealed no graphs in this database that outperform the folded 7-cube. In particular, we prove the following:

\begin{theorem}\label{thm:7}
It holds that    $$
    0.03125 \leq \gamma_7 < 0.0396.
    $$
\end{theorem}

Table \ref{tab:bounds survey}
 gives the current summary of the bounds on $\gamma_k$ which include the contributions from this paper.

\begin{table}[ht!]
    \begin{center}
        \begin{tabular}{|c|c|c|}
            \hline
            Odd girth $k$ & Upper bound for $\gamma_{k}$ & Lower bound for $\gamma_{k}$\\
            \hline
            \hline
            3 & $1$ & $1$ ($K_n, n \rightarrow \infty$)\\
            \hline
            5 & $0.1716$  \cite{C2022} & $0.14$ (Higman-Sims graph \cite{B1998})\\
            \hline
            7 & $0.0396$ (Theorem \ref{thm:7}) & $0.0312$ (folded 7-cube)\\
            \hline
            9 & $0.0396$ (Theorem \ref{thm:7}) & $0.0134$ (9-cycle)\\
            \hline
            11 & $0.0365$ (Proposition \ref{prop:l1+ln bound odd girth through ratio}) & $0.0073$ (11-cycle)\\
            \hline
            13 & $0.0289$ (Proposition \ref{prop:l1+ln bound odd girth through ratio}) & $0.0044$ (13-cycle)\\
            \hline
            15 & $0.0240$ (Proposition \ref{prop:l1+ln bound odd girth through ratio}) & $0.0029$ (15-cycle)\\
            \hline
            \vdots & \vdots & \vdots\\
            \hline
        \end{tabular}
    \end{center}
    \caption{Upper and lower bounds for $\gamma_{k}$ for graphs with a given small odd girth. }
    \label{tab:bounds survey}
\end{table}

The rest of the paper is structured as follows. In Section \ref{sec:preliminaries} we establish some preliminaries and give the reader a background in weight interlacing. In Section \ref{sec:new results}, we prove all our main results. We end in Section \ref{sec:concludingremarks} with some concluding remarks and open problems.

\section{Preliminaries}\label{sec:preliminaries}

In this section we introduce some necessary definitions and results on weight interlacing, which will be the main tool to prove Theorem \ref{thm:7}. For more details on weight interlacing, we refer the reader to \cite{A2019,F1999}.

The idea of weight interlacing is to assign weights originating from eigenvectors in order to ``regularize'' irregular graphs. This has been a powerful tool used generalize results from regular graphs to irregular graphs. Instances of it appear in \cite{F1999}, for extending the celebrated ratio bound on the independence number to irregular graphs, or in \cite{H1995}, where the Hoffman bound on the chromatic number was shown to hold for general graphs. In \cite{F1999}, several other eigenvalue bounds on parameters of regular graphs were also extended to general graphs using weight interlacing, and in \cite{AS2024} a general version of the Expander Mixing Lemma was derived using weight interlacing, illustrating its power.

In order to state the weight interlacing result, we need some definitions. The Perron-Frobenius Theorem (see \cite[Theorem 2.2.1]{spectra}) implies that a connected graph has a unique (up to scaling) eigenvector $\vecnu$ with only positive entries, the \emph{Perron eigenvector}, belonging to the largest eigenvalue of the graph. For $d$-regular graphs, we have $\vecnu=\j$, the all-ones vector. We interpret the entries of the Perron eigenvector as weight for the vertices.

Define $\vecrho : \mathscr P(V(G)) \to \R^{V(G)}$ to be the mapping sending $U\subseteq V(G)$ to
$$\vecrho U=\sum_{u\in U} \nu_u \mathbf{e}_u, $$
with $\mathbf{e}_u$ the unit vector on coordinate $u$. Now if $V(G)=\bigsqcup _{i=1}^t V_i$ is a partition of a graph $G$, write $S$ for the \emph{weight-characteristic matrix}:
\[
S_{u,i} \coloneqq \begin{cases}
    \nu_u &$if $u\in V_i,\\
    0 &$otherwise,$
\end{cases}
\]
for a vertex $u$ and a partition class $V_i$. Note that $S^TS$ is equal to the $t\times t$ diagonal matrix $D$ with entries $\|\vecrho V_i \|^2$. We normalize $S$ to
$$\overline S \coloneqq SD^{-1/2},$$
so that $\overline S^T \overline S$ is the identity matrix. We now recall the weight interlacing result by \cite{F1999}.

\begin{theorem}[Weight interlacing, {\cite[Lemma 2.3(a)]{F1999}}]\label{thm:weight interlacing}
The eigenvalues of
$$B\coloneqq \overline S^TA\overline S,$$
say $\mu_1\ge \dots \geq \mu_n$, \emph{interlace} the eigenvalues $\lambda_1\ge \dots \geq \lambda_n$ of $A$. That is, for every $1\le i \le t$ we have
\[
    \lambda_i \geq \mu_i \text{ and } \lambda_{n+1-i} \leq \mu_{t+1-i}.
\]
\end{theorem}
In particular, the least eigenvalue $\mu_t$ of $B$ is an upper bound for the least eigenvalue $\lambda_n$ of $A$.

It is often more convenient to work not with $B$ itself but with a matrix $M$ that is similar to $B$:
\[
M \coloneqq D^{-1/2} B D^{1/2}=D^{-1} S^TAS.
\]
Since $M$ and $B$ are similar matrices, they have the same eigenvalues so interlacing still applies to $M$ and $A$. Both $M$ and $B$ can be referred to as the \emph{weight-quotient matrix} of the graph with respect to the weights and the partition, but we will reserve this term for the matrix $M$ only.

The entries of $M$ are now given by
\[
M_{ij}=\frac 1{\|\vecrho V_i \|^2}\sum_{u\in V_i}\sum_{v\in V_j} \nu_u \nu_v A_{uv}.
\]
The value $M_{ij}$ can be seen as a weighted average of the weighted number of neighbors in $V_j$ of a vertex in $V_i$. The rows of $M$ sum to $\lambda_1$, which is often helpful for determining the entries of the matrix. This is not necessarily true for $B$.

\section{Proofs of Theorems \ref{thm:general} and \ref{thm:7}}\label{sec:new results}

\subsection{Proof of Theorem \ref{thm:general}: general upper and lower bound for \texorpdfstring{$\gamma_{k}$}{}}

We start with the lower bound of $c_1 k^{-3}$, which, as we mentioned in the introduction, is given by the odd cycle graphs $C_k$ on $k$ vertices (which have odd girth $k$). The spectrum of cycle graphs is well known, see for example \cite[Section 1.4.3]{spectra}. We have
\begin{align*}
\gamma_k &\ge \frac{\lambda_1(C_k)+\lambda_k(C_k)}k=\frac{2\left (1-\cos \left (\frac \pi k\right )\right)}{k},
\intertext{which using a degree 4 Taylor expansion of the cosine function around 0 gives}
\gamma_k &\ge \frac{\pi^2}{k^3}+O \left ( \frac 1 {k^5}\right ),
\end{align*}
giving the lower bound.

For the upper bound we extend the proof idea from \cite{C2022} for (\ref{eq:3-2sqrt2}) to general odd girth. In turn, the argument from \cite{C2022} extends the original argument for (\ref{eq:3-2sqrt2}) for the specific case of regular graphs from \cite{B1998}. Before we prove Theorem \ref{thm:general}, we briefly summarize the arguments from \cite{B1998} and \cite{C2022} that we will  extend in order to derive our bounds.

The argument from \cite{B1998} is as follows. The neighborhood of any vertex in a triangle-free graph is an independent set. In particular, the independence number of a regular triangle-free graph is at least $\lambda_1$. Together with the ratio bound (due to Hoffman, unpublished, cf. \cite{ratiobound}) we get
\begin{equation}\label{eq:ratio bound trianglefree}
    \lambda_1 \le \frac{-\lambda_n n}{\lambda_1 - \lambda_n}.
\end{equation}
Reordering gives an upper bound for $(\lambda_1+\lambda_n)/n$ purely in terms of $\lambda_1/n$, giving the upper bound (\ref{eq:3-2sqrt2}) after a straightforward optimization argument. This argument was extended to irregular graphs in \cite{C2022}, by showing that (\ref{eq:3-2sqrt2}) holds for irregular triangle-free graphs as well.

To prove the upper bound of Theorem \ref{thm:general} we need two preliminary results, namely Lemma \ref{lem:ratio bound odd girth} and Proposition \ref{prop:l1+ln bound odd girth through ratio}. Lemma \ref{lem:ratio bound odd girth} generalizes (\ref{eq:ratio bound trianglefree}) to general odd girth, where the existing result from \cite{C2022} is included as $\ell=1$.

\begin{lemma}\label{lem:ratio bound odd girth}
Let $\ell\ge 1$. If $G$ is a graph on $n$ vertices with odd girth at least $2\ell+3$, then
$$\lambda_{1}\le \frac{-\lambda_{n}^{2\ell-1}n}{\lambda_{1}^{2\ell-1}-\lambda_{n}^{2\ell-1}}.$$
\end{lemma}
\begin{proof}
    Since $G$ has no odd cycles of length at most $2\ell+1$, $G$ does not admit closed walks of length $2\ell+1$. Since $(A^{2\ell+1})_{uv}$ records the number of walks of length $2\ell+1$ between $u$ and $v$, then $A^{2\ell+1}$ only has zeroes on the diagonal and so
    \[0=\tr(A^{2\ell+1})=\sum_{i=1}^n \lambda_i^{2\ell+1}.\]
        By monotonicity of $x \mapsto x^{2\ell-1}$ and by the fact that the trace $A^2$ is equal to twice the number of edges \cite[Proposition 1.3.1]{spectra} we have
        \[0\ge \lambda_1^{2\ell+1}+\lambda_n^{2\ell-1}\sum_{i=2}^n \lambda _i^2=\lambda_1^{2\ell+1}+\lambda_n^{2\ell-1}(2|E(G)|-\lambda_1^2).\]
        Here we can use that $\lambda_1$ is bounded below by the average degree \cite[Proposition 3.1.2]{spectra}, so that
        \[0\ge \lambda_1^{2\ell+1} + \lambda_n^{2\ell-1}(\lambda_1n-\lambda_1^2).\]
    Reordering gives the bound.
\end{proof}

Proposition \ref{prop:l1+ln bound odd girth through ratio} generalizes (\ref{eq:3-2sqrt2}) to general odd girth, where the existing result from \cite{C2022} is included as $\ell=1$.

\begin{proposition}\label{prop:l1+ln bound odd girth through ratio}
    Let $\ell\ge 1$. If $G$ is a graph on $n$ vertices with odd girth at least $2\ell+3$, then
    $$\frac{\lambda_1+\lambda_n}n\le 1-\frac {2\ell}{2\ell-1}x_0,$$
    where $x_0$ is the unique root between $0$ and $1$ of the polynomial $g(x)=x^{2\ell}+2\ell x-2\ell+1$.
\end{proposition}

\begin{proof}
    Rewrite Lemma \ref{lem:ratio bound odd girth} to
    \[
        \lambda_n \le -\lambda_1 \sqrt[2\ell-1]{\frac {\lambda_1}{n-\lambda_1}}.\]
        We substitute $x$ for the root, so that
        \[\frac{\lambda_1+\lambda_n}n \le \frac{\lambda_1-\lambda_1 x}n=\frac{\lambda_1}n (1-x),\]
        and by rewriting $\lambda_1/n$ in terms of $x$ as well, we obtain
        \[\frac{\lambda_1+\lambda_n}n \le \frac{1-x}{\frac 1{x^{2\ell-1}}+1}=\frac{x^{2\ell-1}(1-x)}{1+x^{2\ell-1}}.\]
        Write
        \[f(x)=\frac{x^{2\ell-1}(1-x)}{x^{2\ell-1}+1}.\]
        Note that $f(0)=f(1)=0$, $f(x)>0$ on the interval $(0,1)$, and $f(x)<1$ for $x>1$. To find the optimal value for $f$, we take the derivative:
        \[f'(x) = \frac{-x^{2\ell-2} g(x)}{\left ( x^{2\ell-1}+1\right ) ^2},
        \text{ with }
        g(x)=x^{2\ell}+2\ell x-2\ell+1.\]
        So, on the interval $(0,1)$, we have that $f'(x_0)=0$ if and only if $g(x_0)=0$. On $(0,1)$, $g(x)$ is increasing, and $g(0)<0<g(1)$, so the input $x_0$ maximizing $f(x)$ is uniquely determined. Then, replacing $x_0^{2\ell}$ using that $x_0$ is a root of $g(x)$, we get
        \begin{align*}
            f(x_0)&=\frac{x_0^{2\ell}(1-x_0)}{x_0^{2\ell}+x_0}=\frac{(2\ell-1-2\ell x_0)(1-x_0)}{(2\ell-1)(1-x_0)}\\
        &=\frac{2\ell-1-2\ell x_0}{2\ell-1}=1-x_0-\frac{x_0}{2\ell-1}=1-\frac{2\ell}{2\ell-1}x_0,
    \end{align*}
    which concludes the proof.
\end{proof}

Proposition \ref{prop:l1+ln bound odd girth through ratio} gives an upper bound for $\gamma_k$ for every $k$. For example, for $k=5$ we have $g(x)=x^2+2x-1$, and so $x_0=\sqrt 2-1$, and we get (\ref{eq:3-2sqrt2}). To prove the upper bound of Theorem \ref{thm:general}, we study the asymptotics of the bound of Proposition \ref{prop:l1+ln bound odd girth through ratio} using an estimate. The estimate uses the \emph{Lambert W function} which is the inverse of $x\mapsto xe^x$, where we can and will restrict to the interval $[0,\infty)$. We can now be more specific than in Theorem \ref{thm:general}; we show that
$$\gamma_k < \frac{W(1/e)}{k-4},$$
where $W(1/e)\approx 0.2785$. This implies the upper bound of Theorem \ref{thm:general}.

\begin{proof}[Proof of the upper bound of Theorem \ref{thm:general}]
    We apply Proposition \ref{prop:l1+ln bound odd girth through ratio} with $k=2\ell+3$. We claim that $x_1\coloneqq (\ell-a)/\ell$ is a (strict) lower bound for $x_0$, with $a\coloneqq (W(1/e)+1)/2$. If $x_1$ is indeed a lower bound for $x_0$, then we have
    \begin{align*}
        1-\frac{2\ell}{2\ell-1}x_0&< 1-\frac{2\ell}{2\ell-1}x_1=1-\frac{2\ell-2a}{2\ell-1}=\frac{2a-1}{2\ell-1}=\frac{W(1/e)}{2\ell-1},
        \intertext{which is what we want to prove. In order to see that $x_1< x_0$, consider}
        g(x_1)&=\Big (\frac{\ell-a}\ell \Big )^{2\ell}+2\ell\frac {\ell-a}\ell-2\ell+1=\Big (1- \frac a\ell\Big )^{2\ell}-2a+1.
        \intertext{By $e^x> x+1$ for $x=-a/\ell$ we get}
        g(x_1) &< e^{-2a}-2a+1=e^{-1-W(1/e)}-W(1/e)\\
        &=\frac{1/e}{e^{W(1/e)}}-W(1/e)=\frac{1/e-W(1/e)e^{W(1/e)}}{e^{W(1/e)}}=0,
        \end{align*}
        since $W(x)e^{W(x)}=x$. We now have $g(x_1)< 0$ so indeed $x_1< x_0$, concluding the proof.
\end{proof}

\subsection{Proof of Theorem \ref{thm:7}: upper and lower bound for \texorpdfstring{$\gamma_7$}{}}

Again we start with the lower bound. The \emph{folded 7-cube} is obtained by identifying antipodes of the 7-dimensional hypercube graph. Alternatively, it can be defined as the Cayley graph of the underlying group of $\F_2^6 $ with respect to the generating set consisting of the six unit vectors and the all-ones vector $\j$. A cycle of length 7 is given by
$$0 \sim e_1 \sim e_1+e_2 \sim \dots \sim e_1+e_2+e_3+e_4+e_5+e_6 \sim 0,$$
but a 5-cycle or a 3-cycle are impossible. Using character theory (see \cite[Section 1.4.9]{spectra}) the eigenvalues are easily computed. In particular, the least eigenvalue is $-5$, and so $\gamma_7\ge 1/32 =0.03125$.

We now move to the upper bound from Theorem \ref{thm:7}. As a first step, note that it is enough to consider only connected graphs. Indeed, if $G$ is disconnected, then the spectrum of $G$ is the union of the spectra of its components; if $C$ is a component with $\lambda_1(C)=\lambda_1(G)$, then since $|V(C)|<n$ and $\lambda_{|C|}(C) \ge \lambda_n(G)$ \cite[Proposition 3.2.1(i)]{spectra} we have
$$\frac{\lambda_1(G)+\lambda_n(G)}{n} < \frac{\lambda_1(C)+\lambda_{|C|}(C)}{|V(C)|}.$$
Therefore, without loss of generality, we only consider connected graphs.

To prove the upper bound of Theorem \ref{thm:7} for connected graphs, we use weight interlacing as explained in Section \ref{sec:preliminaries}; let $\vecnu \in \R ^{V(G)}$ be the Perron eigenvector of the graph $G$, and let $\vecrho: \mathcal P(V(G)) \to \R^{V(G)}$ be as in Section \ref{sec:preliminaries}. We use two lemmas concerning the value $\|\vecrho U\|^2$ for different sets of vertices $U$, which could be seen as a ``weighted cardinality'' of the set; note that for regular graphs, we use the all-ones vector $\j$ for the weights, and then $\|\vecrho U \|^2$ is just the cardinality of $U$. This makes the proof of the upper bound of Theorem \ref{thm:7} drastically shorter for regular graphs, but for completeness we prove it in the most general case. In particular, the two lemmas we need are for the irregular case; both of them are trivial in the regular case.

For example, for regular graphs, the cardinality of the neighborhood of any vertex is equal to the largest eigenvalue $\lambda_1$. In terms of the positive constant eigenvector of norm 1 (with entries $1/\sqrt n$), this means that $\| \vecrho N(u) \|^2 = \lambda_1/n$ for every vertex $u$. This generalizes to irregular graphs in the following way.
\begin{lemma}\label{lem:odd girth 7}
    Let $G$ be a connected graph on $n$ vertices with Perron eigenvector $\vecnu $, scaled in such a way that $\|\vecnu \|=1$. Then there exists a vertex $u\in V(G)$ such that
    $$ \|\vecrho N(u)\|^2 \ge \frac {\lambda_1} n.$$
\end{lemma}
\begin{proof}
    By \cite[Proposition 3.1.2]{spectra}, $\lambda_1$ is at least the average degree. So    
    \[
        2|E(G)| \le \lambda_1 n,\]
        and hence by $\| \vecnu \|^2=1$ we have
        \[\sum_{u\in V(G)} \deg(u) \le \lambda_1 n \sum_{u\in V(G)} \nu_u^2,\]
        which implies
        \[\sum_{u\in V(G)} (\lambda_1 n\nu_u^2-\deg (u))\ge 0,\]
        so that there must be a vertex $u$ with $\lambda_1 n \nu_u^2-\deg(u)\ge 0$. In other words,
        \[\frac{\deg(u)}{\nu_u^2}\le \lambda_1 n.\]
        Next, since $x$ is a the Perron eigenvector, we have
        \[\lambda_1 =\sum_{v\in N(u)} \frac{\nu_v }{\nu_u}.\]
        By the Cauchy-Schwarz inequality, we obtain
        \[\lambda_1^2 = \left (  \sum_{v\in N(u)} \frac{\nu_v }{\nu_u}\right )^2 \le \left ( \sum_{v\in N(u)} 1 \right )\left (  \sum_{v\in N(u)} \frac{\nu_v ^2}{\nu_u^2}\right )=\frac{\deg(u)}{\nu_u^2}\|\vecrho N(u)\|^2\]
    The result now follows after applying $\deg(u)/\nu_u^2 \le \lambda_1 n$.
\end{proof}
Secondly, for regular graphs a simple edge counting argument gives that the cardinality of an independent set $S$ is bounded above by the number of vertices that have a neighbor in $S$. This also generalizes to irregular graphs using Perron-eigenvector weights.

\begin{lemma}\label{lem:n/2}
Let $G$ be a connected graph. Let $\vecnu $ be its Perron eigenvector and let $S$ be an independent set in $G$. Write $T$ for the set of vertices that have a neighbor in $S$. Then $\|\vecrho S\|^2 \le \|\vecrho T\|^2$, with equality if and only if $\{S,T\}$ is a bipartition of $G$. In particular, $\|\vecrho S\|^2 \le \frac 12\|\vecnu \|^2$, with equality if and only if $S$ is a bipartite class of $G$.
\end{lemma}
\begin{proof}
    Consider
    \begin{align*}
        \lambda_1 \|\vecrho S\|^2&=\sum_{v\in S} \nu_v  \cdot (\lambda_1 \nu_v ) =\sum_{v\in S} \nu_v  \sum_{w\in N(v)} \nu_w ,
        \intertext{where we use that $\vecnu $ is an eigenvector for the eigenvalue $\lambda_1$. Now note that every neighbor $w$ of $u$ is in $T$, so that}
        &= \sum_{v\in S} \sum_{w\in T} \nu_v  \nu_w  A_{vw}.
        \intertext{We now use that $\vecnu $ is a positive eigenvector and also include the neighbors $v$ of $w\in T$ that are not in $S$, to say}
        &\le \sum_{w\in T} \nu_w  \sum_{v\in N(w)} \nu_v =\sum_{w\in T} \nu_w  (\lambda_1 \nu_w )=\lambda_1 \|\vecrho T\|^2,
    \end{align*}
    from which the inequality follows. We have equality if and only if every neighbor of every vertex in $T$ is in $S$, which is equivalent to $\{S,T\}$ be a bipartition of the graph.
\end{proof}
We are now ready to prove the upper bound of Theorem \ref{thm:7}.
\begin{proof}[Proof of the upper bound of Theorem \ref{thm:7}]
    Since for bipartite graphs $\lambda_1+\lambda_n=0$, we assume without loss of generality that $G$ is not bipartite. Let $\vecnu$ be the Perron eigenvector of $G$, scaled in such a way that $\| \vecnu\|^2=1$. Let $u$ be any vertex of $G$ satisfying Lemma \ref{lem:odd girth 7}. Partition the vertex set $V(G)$ of $G$ into three non-empty parts as follows: Let $V_1$ be the set of vertices at distance $1$ from $u$, let $V_2$ be the set of vertices at distance $0$ or $2$ from $u$, and let $V_3$ the set of vertices at distance at least 3 from $u$. Let $M$ be the $3\times 3$ weight-quotient matrix of the adjacency matrix with respect to this partition
    \[
    M=\begin{bmatrix}
        0 & \lambda_1 & 0\\
        \lambda_1 \frac{\delta}{\alpha} & 0 & \lambda_1\left ( 1-\frac{\delta}\alpha\right )\\
        0 & \lambda_1 \frac{\alpha-\delta}{1-\delta -\alpha} & \lambda_1\left ( 1-\frac{\alpha-\delta}{1-\delta -\alpha}\right )
    \end{bmatrix},
    \]
    where $\delta=\|\vecrho V_1\|^2$ and $\alpha=\|\vecrho V_2 \|^2$. By Theorem \ref{thm:weight interlacing} (weight interlacing) we have
    \[
        \lambda_n(A) \le \lambda_3(M) =-\lambda_1 \left (\frac{\alpha^2-\alpha\delta + \sqrt{\alpha^4+6\alpha^3\delta + 9 \alpha^2 \delta^2-12\alpha^2 \delta-4\alpha\delta^2+4\alpha\delta}}{2\alpha(1-\delta - \alpha)} \right ).\]
        Now
        \[
        \lambda_1+\lambda_n \le \lambda_1 \left (1 - \frac{\alpha^2-\alpha\delta + \sqrt{\alpha^4+6\alpha^3\delta + 9 \alpha^2 \delta^2-12\alpha^2 \delta-4\alpha\delta^2+4\alpha\delta}}{2\alpha(1-\delta - \alpha)} \right ),
    \]
        and by our choice of $u$ satisfying Lemma \ref{lem:odd girth 7}  we have $\lambda_1/n \le \delta$ so that
    \begin{equation}\label{eq:mathematica}
    \frac{\lambda_1+\lambda_n}n \le \delta \left (1-\frac{\alpha^2-\alpha\delta + \sqrt{\alpha^4+6\alpha^3\delta + 9 \alpha^2 \delta^2-12\alpha^2 \delta-4\alpha\delta+4\alpha\delta}}{2\alpha(1-\delta - \alpha)} \right ).
    \end{equation}
    By Lemma \ref{lem:n/2} applied to $S=V_1$ we have $\delta <\alpha$, and by Lemma \ref{lem:n/2} applied to $S=V_2$ and the fact that $\|x\|^2=1$ we have $\alpha <  1/2$. 
    Using Mathematica \cite{Mathematica}, we obtained that on the domain $0<\delta < \alpha < 1/2$, the maximum value of (\ref{eq:mathematica}) is equal to the second largest root of $54x^3+423x^2-700x+27$, which is bounded above by $0.0396$.
\end{proof}

\section{Concluding Remarks}\label{sec:concludingremarks}

The general upper bound we obtained for $\gamma_k$ in Theorem \ref{thm:general} is of magnitude $k^{-1}$ and can likely be improved. However, it is unclear whether the correct order of magnitude for $\gamma_k$ is $k^{-3}$ or if it is strictly larger. While our weight-interlacing method to derive an upper bound on $\gamma_7$ yielded a strong result, generalizing it to larger weight-quotient matrices to obtain a result for $\gamma_k$ quickly becomes saturated (too many variables and too cumbersome even for a computer algebra system). Thus, some new ideas will be needed to extend this result any further. In light of the results obtained in this paper, we propose the following problems for further research.

\begin{problem}
    Determine the correct order of magnitude of $\gamma_k$.
\end{problem}

\begin{problem}
    Even though it is not known whether the Higman-Sims graph attains the largest value of $(\lambda_1+\lambda_n)/n$ among triangle-free regular graphs, let alone among all triangle-free graphs, in the smaller case of triangle-free strongly regular graphs the Higman-Sims graph was shown to be optimal in this regard \cite{C2022}. Does the analogue of this result in the odd girth 7 case hold with respect to the folded 7-cube? That is, does the folded 7-cube attain the largest $\lambda_1 + \lambda_n$ among all $\{ C_3, C_5\}$-free distance-regular graphs of diameter 3?
\end{problem}

\subsection*{Acknowledgments}
Aida Abiad is supported by the Dutch Research Council (NWO) through the grants VI.Vidi.213.085 and OCENW.KLEIN.475. The research of Thijs van Veluw is supported by the Special Research Fund of Ghent University through the grant BOF/24J/2023/047.


\end{document}